\documentclass[a4paper,12pt]{article}

\usepackage{cmlgc}
\usepackage{ucs}

\usepackage[english]{babel}
\usepackage{a4}              
\usepackage[T1]{fontenc}     
\usepackage[utf8]{inputenc}
\usepackage{epsfig}          
\usepackage{amsmath, amsthm} 
\usepackage{amsfonts,amssymb}
\usepackage{accents}

\input{macro}

\usepackage[colorlinks=true]{hyperref}

\newcommand{\N}{\mathbb{N}}

\newcommand{\R}{\mathbb{R}}
\newcommand{\C}{\mathbb{C}}
\newcommand{\D}{\mathbb{C}}

\newcommand{\error}{\mathcal{O}}
\newcommand{\taylor}[1]{\mathcal{T}(#1)}

\title{Smooth normal forms for integrable hamiltonian systems near a
  focus-focus singularity}

\author{San V\~u Ng\d{o}c , Christophe Wacheux\\
Université de Rennes 1}

\begin{document}

\maketitle

\begin{abstract}
  We prove that completely integrable systems are normalisable in the
  $\Cinf$ category near focus-focus singularities.
\end{abstract}

\section{Introduction and exposition of the result}

In his PhD Thesis~\cite{eliasson-these}, Eliasson proved some very
important results about symplectic linearisation of completely
integrable systems near non-degenerate singularities, in the $\Cinf$
category. However, at that time the so-called elliptic singularities
were considered the most important case, and the case of focus-focus
singularities was never published. It turned out that focus-focus
singularities became crucially important in the last 15 years in the
topological, symplectic, and even quantum study of Liouville
integrable systems. The aim of this article is to fill in some
non-trivial gaps in the original treatment, in order to provide the
reader with a complete and robust proof of the fact that a $\Cinf$
completely integrable system is linearisable near a focus-focus
singularity.

Note that in the holomorphic category, the result is already well
established (see Vey~\cite{vey}).

\ouf

Let us first recall the result.

Thoughout the paper, we shall denote by $(q_1,q_2)$ the quadratic
focus-focus model system on $\RM^4=T^*\RM^2$ equipped with the
canonical symplectic form $\omega_0:=d\xi_1\wedge dx_1 + d\xi_2\wedge
dx_2$~:
\begin{equation}
  q_1 := x_1
  \xi_1 + x_2 \xi_2 \quad \text{ and } \quad q_2 := x_1 \xi_2 - x_2
  \xi_1.
  \label{equ:focus-focus}
\end{equation}

Let $f_1,f_2$ be $\Cinf$ functions on a 4-dimensional symplectic
manifold $M$, such that $\{f_1,f_2\}=0$. Here the bracket
$\{\cdot,\cdot\}$ is the Poisson bracket induced by the symplectic
structure. We assume that the differentials $df_1$ and $df_2$ are
independent almost everywhere on $M$. Thus $(f_1,f_2)$ is a completely
(or ``Liouville'') integrable system.

If $m\in M$ is a critical point for a function $f\in\Cinf(M)$, we
denote by $\mathcal{H}_m(f)$ the Hessian of $f$ at $m$, which we view
as a quadratic form on the tangent space $T_mM$.

\begin{defi} For $F=(f_1,f_2)$ an integrable system on a symplectic
  4-manifold $M$, $m$ is a \textbf{critical point of focus-focus type}
  if
  \begin{itemize}
  \item $dF(m)=0$;
  \item the Hessians $\mathcal{H}_m(f_1)$ and $\mathcal{H}_m(f_2)$ are
    linearly independent;
  \item there exist canonical symplectic coordinates on $T_mM$ such
    that these hessians are linear combinations of the focus-focus
    quadratic forms $q_1$ and $q_2$.
  \end{itemize}
\end{defi}
Concretely, this definition amounts to requiring the existence of a
linear symplectomorphism $U:\RM^4\fleche T_mM$ such that~:

\[ U^*(\mathcal{H}_m(F)) = \begin{pmatrix} a q_1 + b q_2 \\ c q_1 + d
  q_2 \end{pmatrix} \; \text{ with } \; \begin{pmatrix} a & b \\ c &
  d \end{pmatrix} \in GL_2(\R). \] From a dynamical viewpoint, this
implies that there exists $(\alpha,\beta)\in\RM^2$, such that the
linearization at $m$ of the hamiltonian vector field associated to the
hamiltonian $\alpha f_1 + \beta f_2$ has four distinct complex
eigenvalues. Thus the Lie algebra spanned by the Hessians of $f_1$ and
$f_2$ is generic (open and dense) within 2-dimensional abelian Lie algebras of
quadratic forms on $(T_mM,\{\cdot,\cdot\})$. This is the non-degeneracy condition 
as defined by Williamson~\cite{williamson}.

The purpose of this paper is to give a complete proof of the following
theorem, which was stated in~\cite{eliasson-these}.
\begin{theo}
  \label{theo-principal}

  Let $(M,\omega)$ be a symplectic 4-manifold, and $F=(f_1,f_2)$ an
  integrable system on $M$ ( \emph{ie.} $\{f_1,f_2\} =0$). Let $m$ be
  a non-degenerate critical point of $F$ of focus-focus type.

  Then there exist a local symplectomorphism $\Psi: (\RM^4,\omega_0)
  \to (M,\omega)$, defined near the origin, and sending the origin to
  the point $m$, and a local diffeomorphism $\tilde{G}: \RM^2\to \RM^2$,
  defined near $0$, and sending $0$ to $F(m)$, such that
  \[ F\circ \Psi = \tilde{G}(q_1,q_2)
  \]
\end{theo}

The geometric content of this normal form theorem becomes clear if,
given any completely integrable system $F$, one considers the
\emph{singular foliation} defined by the level sets of $F$. Thus, the
theorem says that the foliation defined by $F$ may, in suitable
symplectic coordinates, be made equal to the foliation given by the
\emph{quadratic part} of $F$. With this in mind, the theorem can be
viewed as a ``symplectic Morse lemma for singular lagrangian
foliations''.

The normal form $\tilde{G}$ and the normalization $\Psi$ are not
unique. However, the degrees of liberty are well understood. We cite
the following results for the reader's interest, but they are not used
any further in this article.

\begin{theo}[\cite{san-semi-global}]
  If $\phy$ is a local symplectomorphism of $(\RM^4,0)$ preserving the
  focus-focus foliation $\{q:=(q_1,q_2)=\textrm{const}\}$ near the
  origin, then there exists a unique germ of diffeomorphism
  $G:\RM^2\fleche\RM^2$ such that
  \begin{equation}
    q\circ \phy = G\circ q,
  \end{equation}
  and $G$ is of the form $G=(G_1,G_2)$, where
  $G_2(c_1,c_2)=\varepsilon_2 c_2$ and $G_1(c_1,c_2)-\varepsilon_1c_1$
  is flat at the origin, with $\varepsilon_j=\pm1$.
\end{theo}

\begin{theo}[\cite{miranda-zung}]
\label{theo:miranda_zung}
  If $\phy$ is a local symplectomorphism of $(\RM^4,0)$ preserving the
  map $q=(q_1,q_2)$, then $\phy$ is the composition $A\circ \chi$
  where $A$ is a linear symplectomorphism preserving $q$ and $\chi$ is
  the time-1 flow of a smooth function of $(q_1,q_2)$.
\end{theo}

\subsection{ The formulation in complex variables }

Since Theorem~\ref{theo-principal} is a local theorem, we can always
invoke the Darboux theorem and formulate it in local coordinates
$(x_1,\xi_1,x_2,\xi_2)$. Throughout the whole paper, we will switch
whenever necessary to complex coordinates. But here, the complex
coordinates are not defined in the usual way $z = x + i\xi$ but
instead we set $z_1 := x_1 + i x_2$ and $z_2 := \xi_1 + i\xi_2$. We
set then :

\[\label{def-der_compl1}
\frac{\partial }{\partial z_1} := \frac{1}{2} \left(
  \frac{\partial}{\partial x_1} - i \frac{\partial }{\partial x_2}
\right) \ \; , \quad \frac{\partial}{\partial \bar{z_1}} :=
\frac{1}{2} \left( \frac{\partial}{\partial x_1} + i \frac{\partial
  }{\partial x_2} \right) \]

\[\label{def-der_compl2}
\frac{\partial }{\partial z_2} := \frac{1}{2} \left(
  \frac{\partial}{\partial \xi_1} - i \frac{\partial }{\partial \xi_2}
\right) \ \;, \quad \frac{\partial}{\partial \bar{z_2}} := \frac{1}{2}
\left( \frac{\partial}{\partial \xi_1} + i \frac{\partial }{\partial
    \xi_2} \right) \]

Introducing such notation is justified by the next properties :

\begin{prop}
  \begin{itemize}
  \item $q := q_1 + i q_2 = \overline{z}_1 z_2$.
  \item The hamiltonian flows of $q_1$ and $q_2$ in these variables
    are
    \begin{equation}
      \label{eq-flow_C} \varphi_{q_1}^t : (z_1,z_2) \mapsto (e^t z_1,
      e^{-t} z_2) \; \text{ and } \; \varphi_{q_2}^s : (z_1,z_2) \mapsto
      (e^{is} z_1, e^{is} z_2)
    \end{equation}

  \item In complex coordinates, the Poisson bracket for real-valued
    functions is :
    \[ \{f,g\} = 2 \left( -\frac{\partial f}{\partial \bar{z_1}}
      \frac{\partial g}{\partial z_2} + \frac{\partial f}{\partial
        z_2}\frac{\partial g}{\partial \bar{z_1}} - \frac{\partial
        f}{\partial z_1} \frac{\partial g}{\partial \bar{z_2}} +
      \frac{\partial f}{\partial \bar{z_2}} \frac{\partial g}{\partial
        z_1} \right) \]
  \end{itemize}
\end{prop}

\section{Birkhoff normal form for focus-focus singularities}

In this section we show \ref{theo-principal} in a formal context (i.e. : with
formal series instead of functions), and use the formal result to
solve the problem modulo a flat function. For people familiar with
Birkhoff normal forms, we compute here simultaneously the Birkhoff
normal forms of 2 commuting hamiltonians.

\subsection{Formal series}

We consider the space $\RM\formel{x_1,x_2,\xi_1,\xi_2}$ of formal
power expansions in $x_1,x_2,\xi_1,\xi_2$. We recall that this is a
topological space for the $U$-adic topology, where $U$ is the maximal
ideal generated by the variables.

If $\mathring{f}\in \RM\formel{x_1,x_2,\xi_1,\xi_2}$, we write
\[ \mathring{f} = \sum_{N=0}^{+\infty} \mathring{f}^{N} \text{ , with
} \mathring{f}^N = \sum_{i+j+k+l=N} \mathring{f}_{ijkl} \ x_1^i
\xi_1^j x_2^k \xi_2^l \; \text{ and } \mathring{f}_{ijkl} \in \R \]

For $f\in \mathcal{C}^{\infty} (\R^{4},0) $,
$\taylor{f}\in\RM\formel{x_1,x_2,\xi_1,\xi_2}$ designates the Taylor
expansion of $f$ at 0.
We have the following definitions

\begin{defi}
  $\mathring{\error} (N):= \{ \mathring{f} \in
  \R\formel{x_1,x_2,\xi_1,\xi_2} \ | \quad \mathring{f}_{ijkl} = 0,\
  \forall i+j+k+l < N \}$
\end{defi}

\begin{defi}
  \label{de-plate}
  $f\in \mathcal{C}^{\infty} (\R^{4},0) $ is $\error (N)$ (note the
  difference with the previous definition) if one of the 3 equivalent
  conditions is fulfilled :

  \begin{enumerate}
  \item $f$ and all its derivatives of order $< N$ at 0 are 0.

  \item There exists a constant $C_N>0$ such that, in a neighbourhood
    of the origin,
    \begin{equation}
      \abs{f(x_1,x_2,\xi_1,\xi_2)} \leq C_N (x_1^2+x_2^2 +\xi_1^2 +
      \xi_2^2)^{N/2}.
      \label{equ:estimate}    
    \end{equation}

  \item $\taylor{f} \in \mathring{\error} (N)$.
  \end{enumerate}
\end{defi}
The equivalence of the above conditions is a consequence of the Taylor
expansion of $f$. Recall, however, that if $f$ were not supposed to be
smooth at the origin, then the estimates~\eqref{equ:estimate} alone
would not be sufficient for implying the smoothness of $f$.
\begin{defi}
  $f\in \mathcal{C}^{\infty} (\R^{4},0) $ is \emph{flat} at the origin
  or $\error (\infty)$ if for all $N \in \N$, it is $\error(N)$. Its
  Taylor expansion is equally zero as a formal series.
\end{defi}

Smooth functions can be flat and yet non-zero in a neighbourhood of
0. The most classical example is the function $x \mapsto
\exp(-1/x^2)$, which is in fact used in some proofs of the following
Borel lemma.

\begin{lemm}[\cite{borel}]
  \label{lemma-Borel}
  Let $\mathring{f}\in \mathcal{C}^{\infty}(\R^n;\R)
  \formel{X_1,\dots,X_\ell}$.  Then there exists a function $\tilde{f}
  \in \mathcal{C}^{\infty}(\RM^{\ell+n})$ whose Taylor expansion in
  the $(x_1,\dots,x_\ell)$ variables is $\mathring{f}$. This function
  is unique modulo the addition of a function that is flat in the
  $(x_1,\dots,x_\ell)$ variables.
\end{lemm}

We define the Poisson bracket for formal series the same way we do in
the smooth context : for $A,B \in \R\formel{x_1,x_2,\xi_1,\xi_2}$,
\[
\{A,B\} = \sum_{i=1}^n \dfrac{\partial A}{\partial \xi_i} \
\dfrac{\partial B}{\partial x_i} - \dfrac{\partial A}{\partial x_i} \
\dfrac{\partial B}{\partial \xi_i}.
\]

The same notation will designate the smooth and the formal bracket,
depending on the context. We have that Poisson bracket commutes with
taking Taylor expansions~: for formal series $A,B$,
\begin{equation}
  \{\taylor{A},\taylor{B}\} = \taylor{\{A,B\}}.
  \label{equ:taylor-poisson}
\end{equation}

From this we deduce, if we denote $\mathcal{D}_N $ the subspace of
homogeneous polynomials of degree $N$ in the variables
$(x_1,x_2,\xi_1,\xi_2)$:

\[ \{ \error(N),\error(M)\} \subset \error(N+M-2) \; \; \text{ and }
\; \; \{ \mathcal{D}_N, \mathcal{D}_M \} \subset
\mathcal{D}_{N+M-2} \]

We also define $\ad{A} : f \mapsto \{A,f \}$. We still have two
preliminary lemmas needed before starting the actual proof of the
Birkhoff normal form.
\begin{lemm}
  \label{lemma-Taylor-flot}

  For $f \in \Cinf(\RM^4;\R)$ and $A \in \error(3)$ a smooth function,
  we have

  \[
  \taylor{(\varphi^{t}_{A})^* f} = \exp(t\ \ad{\taylor{A}} )f,
  \]
  for each $t\in\RM$ for which the flow on the left-hand side is
  defined.
\end{lemm}
Notice that, since $\taylor{A}\in\mathring{\error}(3)$, the right-hand
side
\[
\exp(t\ \ad{\taylor{A}} )f = \sum_{k=0}^\infty
\frac{t^k}{k!}\left(\ad{\taylor{A}}\right)^k f
\]
is always convergent in $\RM\formel{x_1,x_2,\xi_1,\xi_2}$. In order to
prove this lemma, we shall also use the following result which we
prove immediately~:
\begin{lemm} \label{lemma-error} For $f\in \Cinf(\R^4;\RM^4)$ and $g
  \in \Cinf (\R^4 ; \R)$, if $f(0) = 0$ and $g \in \error(N)$, then $g
  \circ f \in \error(N)$. Moreover if $f$ and $g$ depend on a
  parameter in such a way that their respective
  estimates~\eqref{equ:estimate} are uniform with respect to that parameter, then the corresponding
  $\error(N)$-estimates for $g\circ f$ are uniform as well.
\end{lemm}

\begin{proof}
  Let $\norm{\cdot}_2$ denote the Euclidean norm in $\RM^4$. In view
  of the estimates~\eqref{equ:estimate}, given any two neighbourhoods
  of the origin $U$ and $V$, there exist, by assumption, some
  constants $C_f$ and $C_g$ such that
  \[
  \norm{f(X)}_2 \leq C_f \norm{X}_2 \quad \text{ and } \quad
  \abs{g(Y)} \leq C_g \norm{Y}_2^N,
  \]
  for $X\in U$ and $Y\in V$.  Since $f(0)=0$, we may choose $V$ such
  that $f(U)\subset V$. So we may write
  \[
  \abs{g(f(X))}\leq C_g \norm{f(X)}_2^N\leq C_g C_f^N \norm{X}_2^N,
  \]
  which proves the result.

\end{proof}

\begin{proof}[Proof of Lemma~\ref{lemma-Taylor-flot}]
  We write the transport equation for $f$

\[
\frac{d}{dt} \left( (\varphi^{t}_A)^* f\right)|_{t=s} =
(\varphi^{s}_A)^* \{A , f\}.
\]

When integrated, it comes as

\begin{equation}
  \label{eq-transp_int}
  (\varphi^t_A)^* f = f + \int_0^t (\varphi_A^{s})^* \{A, f \} ds
\end{equation}

We now show by induction that for all $N \in \N$ :

\[ (\varphi_{A}^s)^* f = \sum_{k=0}^N \frac{s^k}{k!} \ (\ad{A})^k (f)
+ \error(N+1), \] uniformly for $s\in[0,t]$.

\paragraph{Initial step $\mathbf{N=0}$ :} Under the integral sign
in~\eqref{eq-transp_int}, $\{A,f\} \in \error(1)$ (at least) and
$\varphi_A^s (0) = 0$, so we have with lemma \ref{lemma-error} that
$\{A, f \} \circ \varphi_A^s \in \error(1)$, uniformly for $s\in
[0,t]$. After integrating the estimate~\eqref{equ:estimate}, the
result follows from~\eqref{eq-transp_int}.

\paragraph{Induction step :} We suppose, for a given $N$ that
\[
(\varphi_{A}^s)^* f = \sum_{k=0}^N \frac{s^k}{k!} \ ad^k_{A}(f)+
\error(N+1),
\]
uniformly for $s\in[0,t]$.

Composing by $\ad{A}$ on the left and by $\varphi_A^s$ on the right,
and then integrating, we get for any $\sigma\in [0,t]$,

\[ (\varphi_A^{\sigma})^* f - f = \int_0^\sigma \frac{d}{ds} \left(
  (\varphi_A^{s})^* f \right) ds = \int_0^\sigma (\varphi^{s}_A)^* \{
A,f \} ds = \int_0^\sigma \{ A, (\varphi^{s}_A)^* f \} ds.
\]
The last equality uses that $\varphi^{s}_A$ is symplectic and
$(\varphi^{s}_A)^*A=A$.  The induction hypothesis gives
\[
\begin{aligned}
  (\varphi^{\sigma}_A)^* f & = f + \int_0^\sigma \sum_{k=0}^N \frac{s^k}{k!} (\ad{A})^{k+1} (f) ds + \int_0^\sigma \ad{A} (\error(N+1)) ds \\
  & = \sum_{k=0}^{N+1} \frac{\sigma^k}{k!} \ (\ad{A})^k(f) +
  \error(N+2),
\end{aligned}
\]
uniformly for $\sigma\in [0,t]$, which concludes the induction.
\end{proof}

The next lemma will be needed to propagate through the induction the
commutation relations.

\begin{lemm}
  \label{lemma-crochet-exp}
  For $f_1,f_2,A$ formal series and $A \in \error(3)$
  \[
  \{\exp(\ad{A}) f_1, \exp(\ad{A}) f_2 \} = \exp(\ad{A}) \{f_1,f_2\}
  \]
\end{lemm}

\begin{proof}
  The Borel lemma gives us $\tilde{A} \in \mathcal{C}^{\infty}$ whose
  Taylor expansion is $A$. Since a hamiltonian flow is symplectic, we
  have
  \begin{equation}
    (\varphi_{\tilde{A}}^t)^* \{f_1,f_2\} = \{
    (\varphi_{\tilde{A}}^t)^* f_1, (\varphi_{\tilde{A}}^t)^* f_2 \}.
    \label{equ:poisson-flot}
  \end{equation}
  We know from Lemma~\ref{lemma-Taylor-flot} that the Taylor expansion
  of $(\varphi_{\tilde{A}}^t)^* f$ is $ \exp(\ad{A})f$. Because the
  Poisson bracket commutes with the Taylor expansion (see
  equation~\eqref{equ:taylor-poisson}), we can simply conclude by
  taking the Taylor expansion in the above
  equality~\eqref{equ:poisson-flot}.
\end{proof}

\subsection{Birkhoff normal form}

We prove here a formal Birkhoff normal form for commuting Hamiltonians
near a focus-focus singularity. Recall that the focus-focus quadratic
forms $q_1$ and $q_2$ were defined in \eqref{equ:focus-focus}.
\begin{theo}
  \label{theo:BirkF}
  Let $f_i\in \R\formel{x_1,x_2,\xi_1,\xi_2}$, $i=1,2$, such that
  \begin{itemize}
  \item $\{f_1,f_2\} = 0$
  \item $f_1 = a q_1 + b q_2 + \error(3)$
  \item $f_2 = c q_1 + d q_2 + \error(3)$
  \item $\begin{pmatrix} a & b \\ c & d \end{pmatrix} \in GL_2(\R)$
  \end{itemize}

  Then there exists $A \in \R\formel{x_1,x_2,\xi_1,\xi_2}$, with $A\in
  \error(3)$, and there exist $g_i\in \R\formel{t_1,t_2}$, $i=1,2$
  such that :
  \begin{equation}
    \exp(\ad{A}) (f_i) = g_i(q_1,q_2), \qquad i=1,2
    \label{equ:birkhoff} 
  \end{equation}
\end{theo}

\begin{proof}

  Let's first start by left-composing $F:=(f_1,f_2)$ by $
  \begin{pmatrix}
    a & b\\c & d
  \end{pmatrix}
  ^{-1}$, to reduce to the case where $f_1 = q_1 + \error(3)$ and $f_2
  = q_2 + \error(3)$. Let $z_1 := x_1 + i x_2$ and $z_2 := \xi_1 + i
  \xi_2$.

  Of course, an element in $\R\formel{x_1,x_2,\xi_1,\xi_2}$ can be
  written as a formal series in the variables
  $z_1,\bar{z_1},z_2,\bar{z_2}$.  We consider a generic monomial $
  {\bf z}^{\alpha \beta} = z_1^{\alpha_1} z_2^{\alpha_2}
  \bar{z_1}^{\beta_1} \bar{z_2}^{\beta_2} $. Using~\eqref{eq-flow_C},
  it is now easy to compute~:
  \begin{equation}
  \begin{aligned}
    \{q_1,  {\bf z}^{\alpha \beta} \} & = \frac{d}{dt} \left(  (\varphi_{q_1}^t)^* z_1^{\alpha_1} z_2^{\alpha_2} \bar{z_1}^{\beta_1} \bar{z_2}^{\beta_2} \right)\vert_{t=0} \\
    & = \frac{d}{dt} \left(e^{(\alpha_1 - \alpha_2 + \beta_1 -
        \beta_2) t} {\bf z}^{\alpha \beta} \right)\vert_{t=0} =
    (\alpha_1 - \alpha_2 + \beta_1 - \beta_2) {\bf z}^{\alpha \beta}.
  \end{aligned}\label{equ:commute-q1}
\end{equation}

  For the same reasons, we have :
  \begin{equation}
    \{q_2, {\bf z}^{\alpha \beta} \} = i(\alpha_1 + \alpha_2 - \beta_1 -
    \beta_2) {\bf z}^{\alpha \beta}.\label{equ:commute-q2}
  \end{equation}

  Hence, the action of $q_1$ and $q_2$ is diagonal on this basis. A
  first consequence of this is the following remark: any formal series
  $f\in \R\formel{x_1,x_2,\xi_1,\xi_2}$ such that $\{f,q_1\} =
  \{f,q_2\}=0$ has the form $f=g(q_1,q_2)$ with $g\in
  \R\formel{t_1,t_2}$ (and the reverse statement is obvious). Indeed,
  let
  \[
  f = \sum f_{\alpha\beta} {\bf z}^{\alpha \beta};
  \]
  
  If $\{q_1,f\} = \{q_2,f\} = 0$, then all $f_{\alpha\beta}$ must vanish, except
  when $\alpha_1 - \alpha_2 + \beta_1 - \beta_2 = 0$ and $\alpha_1 +
  \alpha_2 - \beta_1 - \beta_2 = 0$. For these monomials, $\alpha_1 =
  \beta_2$ and $\alpha_2 = \beta_1$, that is, they are of the form
  $(\bar{z}_1 z_2)^\lambda (z_1 \bar{z}_2)^\mu = q^\lambda
  \bar{q}^\mu$ (remember that $q=q_1 + iq_2 = \bar{z_1} z_2$). Thus we can write
  \[
  f = \sum c_{k\ell} q_1^k q_2^\ell, \qquad c_{k\ell}\in\CM.
  \]
  Specializing to $x_2=0$ we have $q_1^kq_2^\ell =
  x_1^{k+\ell}\xi_1^k\xi_2^\ell$. Since
  $f\in\R\formel{x_1,x_2,\xi_1,\xi_2}$, we see that $c_{k\ell}\in\RM$,
  which establishes our claim.

  We are now going to prove the theorem by constructing $A$ and $g$ by
  induction.

  \paragraph{Initial step : }

  Taking $A=0$, the equation~\eqref{equ:birkhoff} is already satisfied
  modulo $\error(3)$, by assumption.

\paragraph{Induction step :}
Let $N\geq 2$ and suppose now that the equation is solved modulo
$\error(N+1)$~: we have then constructed polynomials
$\mathcal{A}^{(N)}$, $g_i^{(N)}$ such that
\[
\exp(\ad{\mathcal{A}^{(N)}}) (f_i) \equiv \underbrace{g_i^{(N)}
  (q_1,q_2)}_{d^{\circ} < N+1} + \underbrace{ r_i^{N+1}
  (z_1,z_2,\bar{z}_1,\bar{z}_2)}_{\in \mathcal{D}_{N+1} } \mod
\error(N+2).
\]

With the lemma \ref{lemma-crochet-exp}, we have that
\[
\{exp(\ad{\mathcal{A}^{(N)}}) f_1, exp(\ad{\mathcal{A}^{(N)}}) f_2 \}
= exp(\ad{\mathcal{A}^{(N)}}) \{f_1,f_2\} = 0
\]

Hence
\[ \{ g_1^{(N)},g_2^{(N)} \} + \{ g_1^{(N)} , r_2^{N+1} \} + \{
r_2^{N+1},g_1^{(N)} \} + \{ r_1^{N+1} ,r_2^{N+1} \} \equiv 0 \mod
\error(N+2).\]

Since $g_1^{(N)}$ and $g_2^{(N)}$ are polynomials in $(q_1,q_2)$, they
must commute~:\\
$\{ g_1^{(N)},g_2^{(N)} \} =0$. On the other hand, $\{
r_1^{N+1} ,r_2^{N+1} \} \in \mathcal{D} (2N)$ so

\[ \{ \underbrace{g_1^{(N)} }_{=q_1 + \error(3)} , r_2^{N+1} \} + \{
r_1^{N+1}, \underbrace{g_2^{(N)} }_{=q_2 + \error(3)} \} \equiv 0 \mod
\error(N+2)
\]
which implies $\{ r_1^{N+1} , q_2 \} = \{ r_2^{N+1} , q_1 \}$.

One then looks for $\mathcal{A}^{(N+1)} - \mathcal{A}^{(N)} = A_{N+1}
\in \mathcal{D}_{N+1} $, and $g_i^{(N+1)} - g_i^{(N)} = g_i^{N+1}
\in~\mathcal{D}_{N+1} $, with $\{g_i^{N+1},q_j \} = 0$ for $i,j=1,2$,
such that

\begin{equation}
  \exp(\ad{\mathcal{A}^{(N+1)}}) (f_i) \equiv g^{(N+1)} \mod \error(N+2).
  \label{equ:solve}
\end{equation}

We have

\[ \exp(\ad{A^{(N+1)}}) = \exp(\ad{\mathcal{A}^{(N)}} + \ad{A_{N+1}})
= \sum_{k=0}^{+\infty} \frac{( \ad{\mathcal{A}^{(N)}} + \ad{A_{N+1}}
  )^k}{k!} \]

Here, one needs to expand a non-commutative binomial. We set

\[ \langle ( \ad{\mathcal{A}^{(N)}} )^{k-l} (\ad{A_{N+1}})^l
\rangle \]

for the summand of all possible words formed with $k-l$ occurrences of
$\ad{\mathcal{A}^{(N)}}$ and $l$ occurences of $\ad{A_{N+1}}$. We have
then the formula

\[
\begin{split}
  ( \ad{\mathcal{A}^{(N)}} + \ad{A_{N+1}} )^k = \ad{\mathcal{A}^{(N)}}^k & + \langle(\ad{\mathcal{A}^{(N)}})^{k-1} \ad{A_{N+1}} \rangle\\
  & + \ldots + \langle \ad{\mathcal{A}^{(N)}} (\ad{A_{N+1}})^{k-1}
  \rangle + (\ad{A_{N+1}})^k
\end{split}
\]

What are the terms that we have to keep modulo $\error(N+2)$ ? As far
as we only look here to modifications of the total valuation, the
order of the composition between $ad$'s doesn't matter here. Let's
examinate closely :\\
$ \ad{A_{N+1}}^k ( \error(n) ) \subset \error(k(N+1) + n - 2 ) $, and
so for any $h\in\error(3)$ and any $k\geq 1$,
\[
\ad{A_{N+1}}^k(h) \in \error(N+2).
\]

Thus
\[ \exp(\ad{A^{(N+1)}}) (f_i) \equiv \underbrace{\sum_{k=0}^{+\infty}
  \frac{\ad{\mathcal{A}^{(N)}}^k (f_i)}{k!}}_{ =
  exp(\ad{\mathcal{A}^{(N)}}) (f_i) } + \ad{A_{N+1}} (q_i) \mod
\error(N+2). \]

Therefore, equation~\eqref{equ:solve} becomes

\[
g_i^{(N)} + r_i^{N+1} + \ad{A_{N+1}} (q_i) = g^{(N+1)} \mod
\error(N+2),
\]
and hence
\begin{equation}
  \text{\eqref{equ:solve}} \ssi \left\lbrace
    \begin{aligned}
      \{A_{N+1},q_1\} + r_1^{N+1} & = g_1^{N+1}\\
      \{A_{N+1},q_2\} + r_2^{N+1} & = g_2^{N+1}
    \end{aligned}
  \right.
  \label{equ:birkhoff-cohomo}
\end{equation}

Using the notation $A_{N+1}=\sum
A_{N+1,\alpha\beta}\mathbf{z}^{\alpha\beta}$, $r_i^{N+1}=\sum
r_{i,\alpha\beta}^{N+1}\mathbf{z}^{\alpha\beta}$, and $g_i^{N+1}=\sum
g_{i,\alpha\beta}^{N+1}\mathbf{z}^{\alpha\beta}$, we have
\[ \text{\eqref{equ:birkhoff-cohomo}} \ssi \left\lbrace
  \begin{aligned}
    A_{N+1,\alpha\beta}(\alpha_1 - \alpha_2
    + \beta_1 - \beta_2) + r_{1,\alpha\beta}^{N+1} & = g_{1,\alpha\beta}^{N+1}\\
    iA_{N+1,\alpha\beta}(\alpha_1 + \alpha_2 - \beta_1 - \beta_2) +
    r_{2,\alpha\beta}^{N+1} & = g_{2,\alpha\beta}^{N+1}
  \end{aligned}
\right.
\]

We can then solve the first equation by setting,
\begin{itemize}
\item[a)] when $ \alpha_1 - \alpha_2 + \beta_1 - \beta_2 \neq 0 $ :\\
  $g_{1,\alpha\beta}^{N+1}:=0$ and $ A_{\alpha \beta} := \frac{-
    r_{1,\alpha \beta}^{N+1}}{\alpha_1 - \alpha_2 + \beta_1 - \beta_2}
  $ (these choices are necessary);
\item[b)] when $ \alpha_1 - \alpha_2 + \beta_1 - \beta_2 = 0 $ :\\
  $g_{1,\alpha\beta}^{N+1}:=r_{1,\alpha\beta}^{N+1}$ (necessarily),
  and $A_{\alpha \beta} := 0$ (this one is arbitrary).
\end{itemize}

Similarly, we solve the second equation by setting,
\begin{itemize}
\item[c)] when $ \alpha_1 + \alpha_2 - \beta_1 -
  \beta_2 \neq 0 $:\\
  $g_{1,\alpha\beta}^{N+1}:=0$ and $ A_{\alpha \beta} := \frac{-
    r_{2,\alpha \beta}^{N+1}}{i(\alpha_1 + \alpha_2 - \beta_1 -
    \beta_2)} $ (necessarily);
\item[d)] when $ \alpha_1 + \alpha_2 - \beta_1 -
  \beta_2 = 0 $:\\
  $g_{2,\alpha\beta}^{N+1}:=r_{2,\alpha\beta}^{N+1}$ (necessarily),
  and $A_{\alpha \beta} := 0$ (arbitrarily).
\end{itemize}

Notice that (a) and (c) imply (in view of~\eqref{equ:commute-q1}
and~\eqref{equ:commute-q2}) that $g_i^{N+1}$ commutes with $q_1$ and
$q_2$.

Of course we need to check that the choices for $A_{\alpha\beta}$ in
(a) and (c) are compatible with each other. This is ensured by the
``cross commuting relation'' $\{ r_1^{N+1} , q_2 \} = \{ r_2^{N+1} ,
q_1 \} $, which reads
\[
i(\alpha_1 + \alpha_2 - \beta_1 - \beta_2) r_{1,\alpha \beta}^{N+1}
=(\alpha_1 - \alpha_2 + \beta_1 - \beta_2) r_{2,\alpha \beta}^{N+1}.
\]

\end{proof}
\begin{rema}
  The relation $\{ r_1^{N+1} , q_2 \} = \{ r_2^{N+1} , q_1 \} $ is a
  cocycle condition if we look at \eqref{equ:birkhoff-cohomo} as a
  cohomological equation. The relevant complex for this cohomological
  theory is a deformation complex ``\`a la Chevalley-Eilenberg'',
  described, for instance, in \cite{san-miranda}.
\end{rema}

As a corollary of the Birkhoff normal form, we get a statement
concerning $\Cinf$ smooth functions, up to a flat term.

\begin{lemm}
  \label{lemma-BirkP}
  Let $F=(f_1,f_2)$, where $f_1$ and $f_2$ satisfy the same hypothesis
  as in the Birkhoff theorem~\ref{theo:BirkF}. Then there exist a
  symplectomorphism $\chi$ of $\R^4=T^*\RM^2$, tangent to the
  identity, and a smooth local diffeomorphism
  $G:(\R^2,0)\fleche(\R^2,0)$, tangent to the matrix $\begin{pmatrix}
    a & b \\ c & d \end{pmatrix}$ such that~:

  \[ \chi^* F = \tilde{G} (q_1,q_2) + \error (\infty). \]
\end{lemm}

\begin{proof}
  We use the notation of \eqref{equ:birkhoff}. Let $\tilde{g}_j$,
  $j=1,2$ be Borel summations of the formal series $g_j$, and let
  $\tilde{A}$ be a Borel summation of $A$. Let
  $\tilde{G}:=(\tilde{g}_1,\tilde{g}_2)$ and
  $\chi:=\phy^1_{\tilde{A}}$. Applying Lemma~\ref{lemma-Taylor-flot},
  we see that the Taylor series of
  \[
  \chi^*F - \tilde{G}(q_1,q_2)
  \]
  is flat at the origin.

\end{proof}

We can see that Lemma \ref{lemma-BirkP} gives us the main theorem
modulo a flat function. The rest of the paper is devoted to absorbing
this flat function.

\section{A Morse lemma in the focus-focus case}

One of the key ingredients of the proof is a (smooth, but non
symplectic) equivariant Morse lemma for commuting functions. In view
of the Birkhoff normal form, it is enough to state it for flat
perturbations of quadratic forms, as follows.
\begin{theo}
  \label{theo:morse}
  Let $h_1, h_2$ be functions in $\Cinf (\R^4,0)$ such that
  \[ \begin{cases}
    h_1 = q_1 + \error (\infty) \\
    h_2 = q_2 + \error (\infty).
  \end{cases}
  \]
  Assume $\{ h_1 , h_2 \} = 0$ for the canonical symplectic form on
  $\RM^4=T^*\RM^2$.

  Then there exists a local diffeomorphism $\Upsilon$ of $(\R^2,0)$ of
  the form $ \Upsilon = id +\error (\infty) $ such that

  \[ \Upsilon^* h_i = q_i \; , \; i=1,2 \]

  Moreover, we can choose $\Upsilon$ such that the symplectic gradient
  of $q_2$ for $\omega_0$ and for $\omega = \Upsilon^* \omega_0$ are
  equal, which we can write as

  \[(\mathcal{P}): \quad \imath_{X^0_2}(\Upsilon^* \omega_0) = -
  dq_2 \]
 
  That is to say, $\Upsilon$ preserves the $\mathbb{S}^1$-action
  generated by $q_2$.
\end{theo}

\subsection{Proof of the classical flat Morse lemma}

In a first step, we will establish the flat Morse lemma without the
$(\mathcal{P})$ equivariance property. This result, that we call the
``classical`` flat Morse lemma, will be used in the next section to
show the equivariant result.

\begin{proof}
  Using Moser's path method, we shall look for $\Upsilon$ as the
  time-1 flow of a time-dependent vector field $X_t$, which should be
  uniformly flat for $t\in[0,1]$.

  We define : $H_t := (1-t) Q + t H = (1-t)\begin{pmatrix} q_1 \\
    q_2 \end{pmatrix} + t \begin{pmatrix} h_1 \\ h_2 \end{pmatrix}$.
  We want $X_t$ to satisfy

  \[ (\varphi_{X_t}^{t})^* H_t = Q \ ,\ \forall t \in [0,1]. \]
  Differentiating this equation with respect to $t$, we get
  \[
  (\varphi_{X_t}^{t})^* \left[ \frac{\partial H_t}{\partial t} +
    \mathcal{L}_{X_t} H_t \right] = (\varphi_{X_t}^{t})^* \left[-Q + H
    + ( \imath_{X_t} d + d\imath_{X_t}) H_t \right] = 0.
  \]
  So it is enough to find a neighbourood of the origin where one can
  solve, for $t\in[0,1]$, the equation

\begin{equation}
  dH_t(X_t) = Q - H
  =: R.
  \label{equ:moser} 
\end{equation}

Notice that $R$ is flat and $dH_t = dQ - tdR = dQ + \error (\infty)$.

Let $\Omega$ be an open neighborhood of the origin in $\R^4$. Let's
consider $dQ$ as a linear operator from $\mathcal{X}(\Omega)$, the
space of smooth vector fields, to $\Cinf(\Omega)^2$, the space of
pairs of smooth functions. This operator sends flat vector fields to
flat functions.

Before going on, we wish to add here a few words concerning flat
functions. Assume $\Omega$ is contained in the euclidean unit
ball. Let $\Cinf(\Omega)_{\text{flat}}$ denote the vector space of
flat functions defined on $\Omega$. For each integer $N\geq 0$, and
each $f\in \Cinf(\Omega)_{\text{flat}}$, the quantity
\[
p_N(f)=\sup_{z\in\Omega}\frac{\abs{f(z)}}{\norm{z}_2^N}
\]
is finite due to~\eqref{equ:estimate}, and thus the family $(p_N)$ is
an increasing\footnote{$p_{N+1}\geq p_N$} family of norms on
$\Cinf(\Omega)_{\text{flat}}$. We call the corresponding topology the
``local topology at the origin'', as opposed to the usual topology
defined by suprema on compact subsets of $\Omega$. Thus, a linear
operator $A$ from $\Cinf(\Omega)_{\text{flat}}$ to itself is
continuous in the local topology if and only if
\begin{equation}
  \forall N\geq 0, \quad \exists N'\geq 0, \exists C>0, \forall f \quad
  p_{N}(Af)\leq Cp_{N'}(f).
  \label{equ:topology}
\end{equation}

For such an operator, if $f$ depends on an additional parameter and is
uniformly flat, in the sense that the estimates~\eqref{equ:estimate}
are uniform with respect to that parameter, then $Af$ is again uniformly flat.

\begin{lemm}
  \label{lemm:operator}
  Restricted to flat vector fields and flat functions, $dQ$ admits a
  linear right inverse $\Psi : \Cinf(\Omega)^2_{\text{flat}} \fleche
  \mathcal{X}(\Omega)_{\text{flat}}$~: for every $U=(u_1,u_2)\in
  \Cinf(\Omega)^2$ with $u_j\in\error(\infty)$, one has
  \[
  \begin{cases}
    dq_1(\Psi(U)) = u_1 \\
    dq_2(\Psi(U)) = u_2.
  \end{cases}
  \]
  Moreover $\Psi$ is a multiplication operator, explicitly~:
  \begin{equation}
    \Psi(U) = \frac{1}{x_1^2 + x_2^2 + \xi_1^2 + \xi_2^2}
    \begin{pmatrix}
      \xi_1 & \xi_2\\
      \xi_2 & -\xi_1 \\
      x_1 & -x_2\\
      x_2 & x_1
    \end{pmatrix}
    \begin{pmatrix}
      u_1\\ u_2
    \end{pmatrix}
    \label{equ:operator}
  \end{equation}
\end{lemm}
In~\eqref{equ:operator} the right-hand side is a matrix product; we
have identified $\Psi(U)$ with a vector of $4$ functions corresponding
to the coordinates of $\Psi(U)$ in the basis
$(\deriv{}{x_1},\deriv{}{x_2}, \deriv{}{\xi_1},\deriv{}{\xi_2})$.

An immediate corollary of this lemma is that $\Psi$ is continuous in
the local topology. Indeed, if $\abs{u_j}\leq
C(x_1^2+x_2^2+\xi_1^2+\xi_2^2)^{N/2}$ for $j=1,2$, then
\begin{equation}
  \norm{\Psi(U)(x_1,x_2,\xi_1,\xi_2)} \leq d(\Omega)C
  (x_1^2+x_2^2+\xi_1^2+\xi_2^2)^{N/2-1}.
  \label{equ:continuity}    
\end{equation}
Here $\norm{\Psi(U)(x_1,x_2,\xi_1,\xi_2)}$ is the supremum norm in
$\RM^4$ and $d(\Omega)$ is the diameter of $\Omega$.

Now assume the lemma holds, and let $A:=\Psi\circ dR$, where
$R=(r_1,r_2)$ was defined in~\eqref{equ:moser}.  $A$ is a linear
operator from $\mathcal{X}(\Omega)_{\text{flat}}$ to itself, sending a
vector field $v$ to the vector field $\Psi(dr_1(v),dr_2(v))$. From the
lemma, we get~:
\[
dQ(A(v)) = dR(v).
\]
We claim that for $\Omega$ small enough, the operator $(\text{Id}-tA)$
is invertible, and its inverse is continuous in the local topology,
uniformly for $t\in[0,1]$. Now let
\[
X_t:=(\text{Id}-tA)^{-1}\circ \Psi(R).
\]
We compute~:
\[
dH_t(X_t) = dQ(X_t) - tdR(X_t) = dQ(X_t) -t dQ(A(X_t)) =
dQ(\text{Id}-tA)(X_t).
\]
Hence
\[
dH_t(X_t) = dQ (\Psi(R)) = R.
\]
Thus equation~\eqref{equ:moser} is solved on $\Omega$. Since
$X_t(0)=0$ for all $t$, the standard Moser's path argument shows that,
up to another shrinking of $\Omega$, the flow of $X_t$ is defined up
to time $1$. Because of the continuity in the local topology, $X_t$ is
uniformly flat, which implies that the flow at time $1$, $\Upsilon$,
is the identity modulo a flat term.

\end{proof}

To make the above proof complete, we still need to prove
Lemma~\ref{lemm:operator} and the claim concerning the invertibility
of $(\text{Id}-tA)$.

\begin{proof}[Proof of Lemma~\ref{lemm:operator}]

  Let $u:=u_1+ i u_2$ and $q:=q_1 + i q_2$. Thus we want to find a
  real vector field $Y$ such that
  \begin{equation}
    dq(Y) = u.
    \label{equ:moserY}
  \end{equation}

  We invoke again the complex structure used in the Birkhoff theorem,
  and look for $Y$ in the form $Y = a \frac{\partial }{\partial z_1} +
  b \frac{\partial }{\partial \bar{z}_1} + c \frac{\partial }{\partial
    z_2} + d \frac{\partial }{\partial \bar{z}_2} $ . The vector field
  $Y$ is real if and only if $ a = \bar{b} $ and $ c =
  \bar{d}$. Writing
  \[
  dq = d(\bar{z_1} z_2) = z_2 d\bar{z}_1 + \bar{z}_1 dz_2
  \]
  we see that~\eqref{equ:moserY} is equivalent to
  \[
  z_2b + \bar{z}_1c = u.
  \]
  Since $u$ is flat, there exists a smooth flat function $\tilde{u}$
  such that $u = ( |z_1|^2 + |z_2|^2 ) \tilde{u} = z_1 \bar{z}_1
  \tilde{u} + z_2 \bar{z}_2\tilde{u}$.

  Thus we find a solution by letting $\begin{cases} b = \bar{z}_2 \tilde{u} \\
    c = {z}_1 \tilde{u} \end{cases}$. Back in the original basis
  $(\deriv{}{x_1},\deriv{}{x_2}, \deriv{}{\xi_1},\deriv{}{\xi_2})$, we
  have then

  \[ \Psi(U) := Y = \frac{1}{|z_1|^2 + |z_2|^2} \left( \Re e(\bar{z}_2
    u), \Im m(\bar{z}_2 u), \Re e(z_1 u), \Im m(z_1 u) \right)
  \]
  Thus, $\Psi$ is indeed a linear operator and its matrix on this
  basis is
  \[
  \frac{1}{x_1^2 + x_2^2 + \xi_1^2 + \xi_2^2}
  \begin{pmatrix}
    \xi_1 & \xi_2\\
    \xi_2 & - \xi_1 \\
    x_1 & -x_2\\
    x_2 & x_1
  \end{pmatrix}
  \]
\end{proof}

Now consider the operator $A=\Psi\circ dR$.  We see from
equation~\eqref{equ:operator} and the fact that the partial
derivatives of $R$ are flat that, when expressed in the basis
$(\deriv{}{x_1},\deriv{}{x_2}, \deriv{}{\xi_1},\deriv{}{\xi_2})$ of
$\mathcal{X}(\Omega)_{\text{flat}}$, $A$ is a $4\times 4$ matrix with
coefficients in $\Cinf(\Omega)_{\text{flat}}$. In particular one may
choose $\Omega$ small enough such that $\sup_{z\in\Omega}\norm{A(z)}<
1/2$. Thus, for all $t\in [0,1]$, the matrix $(\textup{Id}-t A)$ is
invertible and its inverse is of the form $\textup{Id} +
t\tilde{A}_t$, where $\tilde{A}_z$ has smooth coefficients and
$\sup_{z\in\Omega}\norm{\tilde{A}_t(z)}< 1$. Now $p_N(f+t\tilde{A}_t
f)\leq p_N(f)+p_N(\tilde{A}_tf)\leq 2p_N(f)$~:
$(\textup{Id}+t\tilde{A}_t)$ is uniformly continuous in the local
topology.

With this the proof of Theorem~\ref{theo:morse}, without the
$(\mathcal{P})$ property, is now complete.

\subsection{Proof of the equivariant flat Morse lemma}

We will now use the flat Morse lemma we've just shown to obtain the
equivariant version, \emph{ie.} Theorem~\ref{theo:morse}.

The main ingredient will be the construction of a smooth hamiltonian
$S^1$ action on the symplectic manifold $(\RM^4,\omega_0)$ that leaves
the original moment map $(h_1,h_2)$ invariant. The natural idea is to
define the moment map $H$ through an action integral on the lagrangian
leaves, but because of the singularity, it is not obvious that we get
a smooth function.

Let $\gamma_z$ be the loop in $\RM^{4}$ equal to the
$S^1$-orbit of $z$ for the action generated by $q_2$ with canonical
symplectic form $\omega_0$. Explicitly (see~\eqref{eq-flow_C}), we can
write $z=(z_1,z_2)$ and
\[
\gamma_z(t) = (e^{2\pi it}z_1, e^{2\pi it}z_2), \qquad
t\in[0,1].
\]
Notice that if $z=0$ then the ``loop'' $\gamma_z$ is in fact
just a point. A key formula is
\[
q_2 = \frac{1}{2\pi}\int_{\gamma_z} \!\!\! \alpha_0.
\]
This can be verified by direct computation, or as a consequence of the
following lemma.
 
\begin{lemm}
  \label{lemm:homogeneous}
  Let $\alpha_0$ be the Liouville 1-form on $\RM^{2n}=T^*\RM^n$ (thus
  $\omega_0 = d\alpha_0$). If $H$ is a hamiltonian which is
  homogeneous of degree $n$ in the variables $(\xi_1,\ldots,\xi_n)$,
  defining $\ham{H}^0$ as the symplectic gradient of $H$ induced by
  $\omega_0$, we have :
  \[
  \alpha_0(\ham{H}^0) = nH
  \]
\end{lemm}

\begin{proof}
  Consider the $\R_{+}^*$-action on $T^*\RM^n$ given by multiplication
  on the cotangent fibers:
  $\varphi^t(x_1,\ldots,x_n,\xi_1,\ldots,\xi_n) =
  (x_1,\ldots,x_n,t\xi_1,\ldots,t\xi_n)$.  Since $\alpha_0=
  \sum_{i=1}^n \xi_i dx_i$, we have : $\varphi^{t*} \alpha_0= t
  \alpha_0$.  Taking the derivative with respect to $t$ gives
  $\mathcal{L}_{\varXi} \alpha_0= \alpha_0$ where $\varXi$ is the
  infinitesimal action of $\varphi^t$ :
  $\varXi=(0,\dots,0,\xi_1,\dots,\xi_n)$.  By Cartan's formula,
  $\imath_{\varXi} {d \alpha_0} + d(\imath_{\varXi} \alpha_0) =
  \alpha_0 $. But $\imath_{\varXi} \alpha_0 = \sum_{i=1}^n \xi_i
  dx_i(\varXi) = 0$ so $\alpha_0 = \imath_{\varXi} \omega_0$.

  Thus, $\alpha_0(\ham{H}^0) = \omega_0(\varXi,\ham{H}^0) =
  dH(\varXi)$, and since $H$ is a homogeneous function of degree $n$
  with respect to $\varXi$, Euler's formula gives
  $\mathcal{L}_{\varXi} H = dH(\varXi) = nH$. Therefore we get, as
  required~:
  \[ dH(\varXi) = nH(\varXi) \]
\end{proof}

From this lemma, we deduce, since $q_2$ is invariant under
$\ham{q_2}^0$ :
\[
\frac{1}{2\pi}\int_{\gamma_z} \!\!\! \alpha_0 = \int_0^1 \!\!\!
\alpha_{0 \gamma_2(t)} (\ham{q_2}^0) dt = \int_0^1 \!\!\!  q_2
(\gamma_z(t)) = q_2
\]

By the \emph{classic} flat Morse lemma we have a local diffeomorphism
$\Phi:\RM^4\to\RM^4$ such that $\Phi^*h_j = q_j$, $j=1,2$.  Let
$\alpha:=(\Phi^{-1})^*\alpha_0$, and let
\[
K(z):= \frac{1}{2\pi}\int_{\gamma_z} \!\!\! \alpha,
\]
and let $H:=K\circ \Phi$. Note that
\[
H(m) = \int_{\gamma_2 [m]} \!\!\! \alpha_0
\]
where $\gamma_2[m]:=\Phi^{-1}\circ\gamma_2[z=\Phi(m)]$, and
$H(0)=0$.

We prove now that $H$ is a hamiltonian moment map for an $S^1$-action
on $M$ leaving $(h_1,h_2)$ invariant.

\begin{lemm}
  $H\in\Cinf(\RM^4,0)$.
\end{lemm}

\begin{proof} Equivalently, we prove that $K\in\Cinf(\RM^4,0)$.  The
  difficulty lies in the fact that the family of ``loops''
  $\gamma_z$ is not locally trivial: it degenerates into a
  point when $z=0$. However it is easy to desingularize $K$, as
  follows. Again we identify $\RM^4$ with $\CM^2$; we introduce the
  maps~:

  \[ \begin{aligned}
    j:& \; \D \times \CM^2 \! \!& \longrightarrow & ~\CM^2                  & \; &    j_z: &\D & \to M   \\
    \! & (\zeta,(z_1,z_2)) \! \!& \mapsto & (\zeta z_1, \zeta z_2) &
    \; & \;\; &\zeta & \mapsto (\zeta z_1, \zeta z_2)
  \end{aligned}
  \]
  so that $\gamma_z = (j_z)_{\restr U(1)}$. Let $D\subset\CM$
  be the closed unit disc $\{\zeta\leq 1\}$. Thus
  \[
  \int_{\gamma_z} \alpha = \int_{j_z(U(1))}\alpha = \int_{U(1)}
  j_z^*\alpha = \int_{\partial D} j_z^*\alpha.
  \]
  Let $\omega:=d\alpha$. By Stokes' formula,
  \[
  \int_{\partial D} j_z^*\alpha = \int_D \!\!\!  j_z^* \omega = \int_D
  \!\!\! \omega_{j(z,\zeta)} (d_{\zeta} j(z,\zeta) (\cdot),d_{\zeta}
  j(z,\zeta) (\cdot)).
  \]
  Since $D$ is a fixed compact set and $\omega$, $j$ are smooth, we
  get $K\in\Cinf(\RM^4,0)$.
\end{proof}

Consider now the integrable system $(h_1,h_2)$.  Since $H$ is an
action integral for the Liouville 1-form $\alpha_0$, it follows from
the action-angle theorem by Liouville-Arnold-Mineur that the
hamiltonian flow of $H$ preserves the regular Liouville tori of
$(h_1,h_2)$. Thus, $\{H,h_j\}=0$ for $j=1,2$ on every regular
torus. The function $\{H,h_j\}$ being smooth hence continuous,
$\{H,h_j\}=0$ everywhere it is defined~: $H$ is locally constant on
every level set of the joint moment map $(h_1,h_2)$. Equivalently, $K$
is locally constant on the level sets of $q=(q_1,q_2)$. It is easy to
check that these level sets are locally connected near the
origin. Thus there exists a map $g:(\RM^2,0)\fleche \RM$ such that
\[
K = g\circ q.
\]
It is easy to see that $g$ must be smooth~: indeed, $K$ itself is
smooth and, in view of \eqref{equ:focus-focus}, one can write
\begin{equation}
  g(c_1,c_2) = K(x_1=c_1,x_2=-c_2,\xi_1=1,\xi_2=0).
  \label{equ:g-smooth}
\end{equation}
We claim that the function $(c_1,c_2)\mapsto g(c_1,c_2)-c_2$ is flat
at the origin~: since $\Phi = id + \error (\infty)$, we have~: $
\alpha = \Phi^* \alpha_0 = \alpha_0 + \error (\infty)$ so $K(z) =
\int_{\gamma_2 [z] } \alpha_0 + \error (\infty) = q_2 + \eta(z)$
with $\eta$ a flat function of the 4 variables. We show now the
lemma~:
\begin{lemm}
  \label{lemm:mu-flat}
  Let $\eta \in \Cinf(\R^4; \R)$ be a flat function at the origin in
  $\RM^4$ such that $\eta(z) = \mu(q_1,q_2)$ for some map
  $\mu:\RM^2\fleche\RM$. Then $\mu$ is flat at the origin in $\RM^2$.
\end{lemm}
\begin{proof}
  We already know from~\eqref{equ:g-smooth} that $\mu$ has to be
  smooth. Thus, it is enough to show the estimates~: $\forall N \in
  \N, \exists C_N \in \R$ such that
  \[
  \forall (c_1,c_2) \in \R^2, \qquad \abs{\eta(c_1,c_2)} \leq C_N
  \norm{(c_1,c_2)}^N = C_N (\abs{c_1}^2 + \abs{c_2}^2)^{N/2}.
  \]
  Since $\eta$ is flat, we have, for some constant $C_N$,
  \[
  \abs{\eta(x_1,x_2,\xi_1,\xi_2)}\leq {C_N}
  \norm{(x_1,x_2,\xi_1,\xi_2)}^N.
  \]
  But for any $c=(c_1+ic_2)\in\CM$, there exists $(z_1,z_2)\in
  q^{-1}(c)$ such that $2\abs{c}^2=\abs{z_1}^2+\abs{z_2}^2$~: if $c=0$
  we take $z=0$, otherwise take $z_1:=\abs{c}^{1/2}$ and $z_2:=c/z_1$,
  so that $q(z_1,z_2)=\bar{z}_1z_2 = c$.

  Therefore, for all $(c_1,c_2) \in \R^2$ we can write
  \[
  |\mu(c_1,c_2)| = |\eta(z_1,z_2) | \leq C_N \| (z_1,z_2) \|^N \leq 2
  C_N |c|^N,
  \]
  which finishes the proof.
\end{proof}

We have now :

\[ \begin{pmatrix} h_1 \\ H \end{pmatrix} = \begin{pmatrix} h_1 \\ h_2
  + \mu(h_1,h_2) \end{pmatrix} \]

By the implicit function theorem, the function $V:(x,y) \mapsto (x,y +
\mu(x,y)$ is locally invertible around the origin, since $\mu$ is flat; moreover,
$V^{-1}$ is infinitely tangent to the identity. Therefore, in view of
the statement of Theorem~\ref{theo:morse}, we can replace our
initial integrable system $(h_1,h_2)$ by the system
$V\circ(h_1,h_2)=(h_1,H)$.

Thus, we have reduced our problem to the case where $h_2=H$ is a
hamiltonian for a smooth $S^1$ action on $\RM^4$. We denote by $S^1_H$
this action. The origin is a fixed point, and we denote by
$\textup{lin}(S^1_H)$ the action linearized at the origin. We
now invoke an equivariant form of the Darboux theorem.
\begin{theo}[Darboux-Weinstein~\cite{chaperon-asterique}]

  There exists $\varphi$ a diffeomorphism of $(\R^4,0)$ such that :

  \[ \left(\R^4,\omega_0,{S}_H^1 \right) \xrightarrow{\varphi}
  \left(T_0\R^4,T_0\omega_0,\textup{lin}({S}_H^1) \right) \]
 
\end{theo}

The linearization $T_0 \omega_0$ of $\omega_0$ is $\omega_0$ :
$\varphi$ is a symplectomorphism, and the linearization of the
${S}^1$-action of $H$ is the ${S}^1$-action of the quadratic part of
$H$, which is $q_2$. Hence $H\circ\phy^{-1}$ and $q_2$ have the same
symplectic gradient, and both vanish at the origin~: so $H \circ
\varphi^{-1} = q_2$. So we have got rid of the flat part of $h_2$
without modifying the symplectic form. The last step is to give a
precised version of the equivariant flat Morse lemma~:

\begin{lemm}
  \label{lemma-Morseplat-precise}
  Let $h_1, h_2$ be functions in $\Cinf (\R^4,0)$ such that
  \[ \begin{cases}
    h_1 = q_1 + \error (\infty) \\
    h_2 = q_2
  \end{cases}
  \]

  Then there exists a local diffeomorphism $\Upsilon$ of $(\R^2,0)$ of
  the form $ \Upsilon = id +\error (\infty) $ such that
  \[
  \Upsilon^* h_i = q_i \; , \; i=1,2
  \]

  Moreover, we can choose $\Upsilon$ such that the symplectic gradient
  of $q_2$ for $\omega_0$ and for $\Upsilon^* \omega_0$ are equal,
  which we can write as
  \[
  (\mathcal{P}): \; \; \imath_{\ham{q_2}^0}(\Upsilon^* \omega_0) = -
  dq_2
  \]
  That is, $\Upsilon$ preserves the $S^1$-action generated by $q_2$.
\end{lemm}

\begin{proof}
  Following the same Moser's path method we used in the proof of the
  classical flat Morse lemma, we come up with the following
  cohomological equation~\eqref{equ:moser}
  \[
  (Z) \begin{cases}
    (dq_1 +t dr_1) (X_t) = r_1 \\
    dq_2(X_t) = 0
  \end{cases}
  \]
  The classical flat Morse lemma ensures the existence of a solution
  $X_t$ to this system. We have then that $\{ r_1,q_2 \} = \{ r_2 ,
  q_1 \} = 0$, because here $r_2 = 0$. We have also that : $\{ q_1,q_2
  \} = 0$, $\{ q_2,q_2 \} = 0$ , so $r_1$,$q_1$ and $q_2$ are
  invariant by the flow of $q_2$. So we can average $(Z)$ by the
  action of $q_2$~: let $\phy_2^s:=\phy_{\ham{q_2}^0}^s$ be the time
  $s$-flow of the vector field $\ham{q_2}^0$ and let
  \[
  \crochet{X_t} := \frac{1}{2\pi}\int_0^{2\pi} (\phy_2^s)^* X_t ds.
  \]
  If a function $f$ is invariant under $\phy_2^s$, \emph{i.e.}
  $(\phy_2^s)^*f = f$, then
  \[
  ((\phy_2^s)^*X_t)f= ((\phy_2^s)^*X_t)((\phy_2^{s})^*f) =
  (\phy_2^s)^*(X_tf).
  \]
  Integrating over $s\in[0,2\pi]$, we get $\crochet{X_t}f =
  \crochet{X_tf}$, where the latter is the standard average of
  functions.  Therefore $\crochet{X_t}$ satisfies the system (Z) as
  well.

  Finally, we have, for any $s$, $(\varphi_2^{t})^* \crochet{X_t} =
  \crochet{X_t}$ which implies
  \begin{equation}
    [\ham{q_2}^0,\crochet{X_t}]=0;
    \label{equ:crochet} 
  \end{equation}
  in turn, if we let $\phy_{\crochet{X_t}}^t$ be the flow of the
  non-autonomous vector field $\crochet{X_t}$,
  integrating~\eqref{equ:crochet} with respect to $t$ gives
  $(\phy_{\crochet{X_t}}^t)^*\ham{q_2}^0=\ham{q_2}^0$. For $t=1$ we
  get $\Upsilon^* \ham{q_2}^0=\ham{q_2}^0$. But, by naturality
  $\Upsilon^* \ham{q_2}^0$ is the symplectic gradient of
  $\Upsilon^*q_2=q_2$ with respect to the symplectic form
  $\Upsilon^*\omega_0=\omega$, so property $(\mathcal{P})$ is
  satisfied.
\end{proof}

\section{Principal lemma }

\subsection{Division lemma}

The following cohomological equation, formally similar to
to~\eqref{equ:birkhoff-cohomo}, is the core of
Theorem~\ref{theo-principal}.
\begin{theo}
  \label{theo:division}
  Let $r_1, r_2 \in \mathcal{C}^{\infty} ((\R^4,0);\RM)$, flat at the
  origin such that $\{r_1,q_2\} = \{r_2,q_1\}$.  Then there exists $f
  \in \mathcal{C}^{\infty} ( (\R^4,0); \RM )$ and $\phi_2 \in \Cinf
  ((\R^2,0); \RM ) $ such that

  \begin{equation}
    \label{eq-cohom}
    \begin{cases}
      \{ f, q_1 \} (x,\xi) =  r_1 \\
      \{ f, q_2 \} (x,\xi) = r_2 - \phi_2(q_1,q_2),
    \end{cases}
  \end{equation}
  and $f$ and $\phi_2$ are flat at the origin.  Moreover $\phi_2$ is
  unique and given by
  \begin{equation}
    \forall z\in \RM^4, \quad \phi_2(q_1(z),q_2(z)) = \frac{1}{2\pi}
    \int_0^{2\pi} (\varphi_{q_2}^{s})^* r_2(z) ds,
    \label{equ:phy-average}
  \end{equation}
  where $s\mapsto\phy_{q_2}^s$ is the hamiltonian flow of $q_2$.
\end{theo}

One can compare the difficulty to solve this equation in the 1D
hyperbolic case with elliptic cases. If the flow is periodic, that is,
if $SO(q)$ is compact, then one can solve the cohomological equation
by averaging over the action of $SO(q)$. This is what happens in the
elliptic case. But for a hyperbolic singularity, $SO(q)$ is not
compact anymore, and the solution is more technical
(see~\cite{colin-vey}). In our focus-focus case, we have to solve
simultaneously two cohomological equations, one of which yields a
compact group action while the other doesn't. This time again, it is
the ``cross commuting relation'' $\{r_i,q_j\} = \{r_j,q_i\}$ that we
already encountered in the formal context that will allow us to solve
simultaneously both equations.

\begin{proof}

  Let $\varphi_{q_1}^s, \varphi_{q_2}^s$ be respectively the flows of
  $q_1$ and $q_2$. From \eqref{eq-flow_C}, we see $\varphi_{q_2}$ is
  $2\pi$-periodic : $q_2$ is a momentum map of a hamiltonian
  $\mathbb{S}^1$-action. Since
  $(\phy_{q_2}^s)^*\{f,q_2\}=-\frac{d}{dt}\left[(\phy_{q_2}^s)^*f \right]$, the
  integral of $\{f,q_2\}$ along a periodic orbit vanishes,
  and~\eqref{equ:phy-average} is a necessary consequence
  of~\eqref{eq-cohom}.  We set, for all $z\in\RM^4$ :
  \[
  h_2(z) = \frac{1}{2\pi} \int_0^{2\pi} (\varphi_{q_2}^{s})^* r_2(z)
  ds.
  \]
  Notice that $h_2$ is smooth and flat at the origin.  One has
  $\{h_2,q_2\} = 0 $ and
  \[
  \begin{aligned}
    \{h_2,q_1\} & = \frac{1}{2\pi} \int_0^{2\pi} \{ (\varphi_{q_2}^{s})^*r_2,q_1 \} ds = \frac{1}{2\pi} \int_0^{2\pi} \{ (\varphi_{q_2}^{s})^*r_2, (\varphi_{q_2}^{s})^* q_1 \} ds \\
    & = \frac{1}{2\pi} \int_0^{2\pi} (\varphi_{q_2}^{s})^*
    \underbrace{\{ r_2,q_1 \}}_{= \{r_1,q_2\}} ds = \left\lbrace
      \frac{1}{2\pi} \int_0^{2\pi} (\varphi_{q_2}^{s})^* r_1 ds , q_2
    \right\rbrace = 0
  \end{aligned}
  \]

  Thus $dh_2$ vanishes on the vector fields $\ham{q_1}$ and
  $\ham{q_2}$, which implies that $h_2$ is locally constant on the
  smooth parts of the level sets $(q_1,q_2)=\text{const}$. As before
  (Lemma~\ref{lemm:mu-flat}) this entails that there is a unique germ
  of function $\phi_2$ on $(\RM^2,0)$ such that $h_2= \phi_2
  (q_1,q_2)$; what's more $\phi_2$ is smooth in a neighbourhood of the
  origin, and flat at the origin.

  Next, we define $\check{r}_2 = r_2 - h_2$ and
  \[
  f_2(z) = -\frac{1}{2\pi} \int_0^{2\pi} s(\varphi_{q_2}^{s})^* (
  \check{r}_2(z) ) ds.
  \]
  Then $f_2\in \mathcal{C}^{\infty} ((\R^4,0);\RM)_{\text{flat}}$, and
  we compute~:
  \begin{align*}
    \{q_2,f_2 \} & = -\frac{1}{2\pi} \int_0^{2\pi} s \{q_2,(\varphi_{q_2}^{s})^* \check{r}_2 \} ds \\
    & = -\frac{1}{2\pi} \int_0^{2\pi} s (\varphi_{q_2}^{s})^*
    \{q_2,\check{r}_2\} ds
    = -\frac{1}{2\pi} \int_0^{2\pi} s \frac{d}{ds} ((\varphi_{q_2}^{s})^* \check{r}_2 ) ds \\
    & = -\frac{1}{2\pi} \left[ s(\varphi_{q_2}^{s})^* \check{r}_2
    \right]_{s=0}^{s=2\pi} + \frac{1}{2\pi} \underbrace{\int_0^{2\pi}
      (\varphi_{q_2}^{s})^* \check{r}_2 ds}_{=0} = -\check{r}_2 =
    \phi_2(q_1,q_2) - r_2.
  \end{align*}

  In the last line we have integrated by parts and used that
  $\check{r}_2$ has vanishing $\mathbb{S}^1$-average.  Hence $f_2$ is
  solution of $\{f_2,q_2\} = r_2 - \phi_2(q_1,q_2)$ and $\{ f_2 , q_1
  \} = 0$. So
  \[
  \text{\eqref{eq-cohom}} \Leftrightarrow \begin{cases} \{f - f_2, q_1
    \} = r_1 \\ \{f-f_2,q_2\} =0 \end{cases}
  \]
  So we managed to reduce the initial cohomological
  equations~\eqref{eq-cohom} to the case $r_2 = \phi_2 = 0$; In other
  words, upon replacing $f-f_2$ by $f$, we need to solve
  \begin{equation}
    \label{equ:eq-cohom-bis}
    \begin{cases} \{f , q_1 \}
      = r_1 \\ \{f,q_2\} =0 \end{cases}
  \end{equation}
  and the cocycle condition simply becomes $\{r_1,q_2\} =0$.

  Now, in order to solve the first cohomological equation
  of~\eqref{equ:eq-cohom-bis}, we shall reduce our problem to the case
  where $r_1$ is a flat function on the planes $z_1 =0$ and $z_2 =
  0$. For this we shall adapt the technique used by Colin de Verdière
  and Vey in \cite{colin-vey}, but in a 4-dimensional setting.

  Let $z_1 = r e^{it}$ and $z_2 = \rho e^{i\theta}$. We introduce the
  usual vector fields~:
  \[
  r\deriv{}{r} = {x_1} \deriv{}{x_1} + {x_2} \deriv{}{x_2} \; , \qquad
  \deriv{}{t} = - x_2 \deriv{}{x_1} + x_1 \deriv{}{x_2} \]
  \[ \rho\deriv{}{\rho} = {\xi_1} \deriv{}{\xi_1} + {\xi_2}
  \deriv{}{\xi_2} \; , \qquad \deriv{}{\theta} = - \xi_2
  \deriv{}{\xi_1} + \xi_1 \deriv{}{\xi_2}.
  \]
  Remember that in complex notation $q = q_1 + i q_2 =
  \bar{z_1}z_2$. In polar coordinates, our system becomes
  \begin{equation}
    \label{equ:eq-cohom_pol}
    \text{\eqref{equ:eq-cohom-bis}} \Leftrightarrow 
    \begin{cases} r \deriv{f}{r} - \rho \deriv{f}{\rho} = r_1\\
      \left( \deriv{f}{t} + \deriv{f}{\theta} \right) = 0,
    \end{cases}
  \end{equation}
  and the cross-commuting relation becomes: $\left( \deriv{}{t} +
    \deriv{}{\theta} \right) r_1 = 0$.

  We first determine the Taylor series of $f$ along the $z_2 =0$
  plane; we denote for any function $h \in \mathcal{C}^{\infty} (\R^4;
  \R)$
  \[ [h]_{\ell_1,\ell_2} (x_1,x_2) := \frac{\partial^{\ell_1+\ell_2} h}{\partial
    \xi_1^{\ell_1} \partial \xi_2^{\ell_2}} (x_1,x_2,0,0).
  \]
  Now, when applying $ \ell_1$ times $\deriv{}{\xi_1}$ and $\ell_2$
  times $\deriv{}{\xi_2}$ to the equations~\eqref{equ:eq-cohom_pol},
  and then setting $\xi_1 = \xi_2 = 0$, one gets
  \begin{equation*}
    \begin{cases}
      r \deriv{[f]_{\ell_1,\ell_2}}{r} (x_1,x_2) - (\ell_1 + \ell_2 ) [f]_{\ell_1,\ell_2} (x_1,x_2) = [r_1]_{\ell_1,\ell_2} (x_1,x_2) \\
      \deriv{[f]_{\ell_1,\ell_2}}{t} = 0 \\
      \deriv{[r_1]_{\ell_1,\ell_2}}{t} = 0 \\
    \end{cases}
  \end{equation*}

  So the second and third equation tells us that
  $[r_1]_{\ell_1,\ell_2}$ and $[f]_{\ell_1,\ell_2}$ can be written as
  continuous functions of
  $r$
  . We set $[r_1]_{\ell_1,\ell_2} (x_1,x_2) = [R_1]_{\ell_1,\ell_2}
  (\sqrt{x_1^2 + x_2^2})$ and $[f]_{\ell_1,\ell_2} (x_1,x_2) =
  [F]_{\ell_1,\ell_2} (\sqrt{x_1^2 + x_2^2})$. The equation satisfied
  by $[F]_{\ell_1,\ell_2}$ and $[R_1]_{\ell_1,1_2}$ is actually an
  ordinary differential equation of the real variable $r$, which
  admits as a solution

  \[ [F]_{\ell_1,\ell_2} (r) = \int_0^1 t^{-(\ell_1+\ell_2+1)}
  [R_1]_{\ell_1,\ell_2} (t r) dt \] For any $t\geq 0$,
  $[R_1]_{\ell_1,1_2}(t)=[r_1]_{\ell_1,\ell_2} (t,0)$; hence
  $[R_1]_{\ell_1,1_2}$ is smooth on $\RM^+$, and flat at the
  origin. This implies that the above integral is convergent for any
  $(\ell_1,\ell_2)$, and defines a smooth function of $r\geq 0$, which is also flat when $r \to 0^+$.  We
  shall now check that $[F]_{\ell_1,\ell_2} (\sqrt{x_1^2 + x_2^2})$ is
  a smooth (and flat) as a function of $x_1$ and $x_2$. Obviously, the
  only problem that can occur is at $(0,0)$, since the square root is
  not a smooth function at the origin. Yet, we can see with the help
  of the Faa Di Bruno formula that, for any smooth function $F$ on
  $(0,\infty)$,

\[
\frac{d^n}{dx^n} \left[ F (\sqrt{x}) \right] = \sum_{k=0}^n
F^{(n)}(\sqrt{x})
B_{n,k}(\frac{d}{dx}(\sqrt{x}),\frac{d^2}{dx^2}(\sqrt{x}),\ldots,\frac{d^n}{dx^n}
(\sqrt{x})).
\]

(Here $B_{n,k}$ designate the $(n,k)$-th Bell polynomial; of course we
don't need its exact value). Since $\frac{d^n}{dx^n} (\sqrt{x}) =
\frac{(-1)^n (2n-1)!}{2^{2n-1} (n-1)!} x^{\frac{1}{2}-n} $, we have in
fact that $\frac{d^n}{dx^n} \left[ F \circ (\sqrt{x}) \right]$ is a
finite sum of terms of the form

\[ C_j \frac{F^{(j)} (\sqrt{x})} {x^{\frac{N_j}{2}}} \; \text{ with }
C_j \in \R \text{ and } N_j \in \N. \]

Since, in our case, $F$ is in fact flat when $r \to 0^+$, all of these terms tends
to $0$ when $x$ does, so $F \circ \sqrt{.}$ is a smooth, and actually
flat function. Hence $[f]_{\ell_1,\ell_2}$ is a flat function of the
variables $(x_1,x_2)$. 
Now, invoking Borel's lemma, let $u_1 \in \mathcal{C}^{\infty} (\R^4;
\R)$, whose Taylor expansion in the $\xi$ variables is
\[ \mathcal{T}_{(\xi_1,\xi_2)} {u_1} = \sum_{\ell_1,\ell_2} \frac{1}{
  \ell_1!  \ell_2!} [f]_{\ell_1,\ell_2} (x_1,x_2) \xi_1^{\ell_1}
\xi_2^{\ell_2}. \]

We have that $u_1$ is flat all along the axis $z_1 =0$. One can always
symmetrize this last function by the action of $q_2$. We still get a
flat function on $z_1 =0$. Let's review the properties of $u_1$

\begin{itemize}
\item $r_1 - \{u_1,q_1\}$ is flat on $z_2 = 0$ by construction of
  $u_1$
\item $u_1$ and therefore $\{u_1,q_1\}$ are flat on $z_1 =
  0$
\item $\{u_1,q_2\} = 0$ on the whole space.
\end{itemize}

On can construct by the same process $u_2$, flat all along $z_2 =0$
and such that $ \{u_2,q_1\} - r_1$ is flat on $z_1 = 0$, and $
\{u_2,q_2\} = 0$ the whole space. The cohomological
equations~\eqref{equ:eq-cohom-bis} are now equivalent to
\[
\begin{cases}
  \{ f- u_1 - u_2 , q_1 \} = r_1 - \{u_1+u_2,q_1\} \\
  \{f - u_1 - u_2 , q_2 \} = - \{u_1+u_2,q_2\} = 0.
\end{cases}
\]
Notice that $r_1 - \{u_1+u_2,q_1\} =r_1 - \{u_1,q_1\} - \{u_2,q_1\}$
is flat on \emph{both} planes $z_1 = 0$ and $z_2 = 0$.  Thus,
replacing $f-u_1-u_2$ by $f$ again, we are now reduced to the case
where $r_1$ is a flat function on both planes $z_1=0$ and
$z_2=0$. We'll now solve the cohomological equation away from $z_1 =
0$ or $z_2 =0$ and extend it to these planes. 
If $f$ is a solution, it satisfies the transport equation
\[ -\frac{d}{dt} \left.\left[ (\varphi_{q_1}^{t})^* f
  \right]\right|_{t=s} = (\varphi_{q_1}^{s})^* \{f,q_1\} =
(\varphi_{q_1}^{s})^* r_1.
\]
Integrating between $0$ and a time $T$ depending on $z_1$ and $z_2$,
we get
\[ f - (\varphi_{q_1}^{T})^* f = \int_0^T (\varphi_{q_1}^{s})^* r_1
ds. \]

Looking for a solution $f$ of the form $f = \int_0^T
(\varphi_{q_1}^{s})^* r_1 ds$, we get
\[
\{f,q_1\} = r_1 - (1+\{q_1,T\})(\phy_{q_1}^T)^* r_1
\]
and, using $\{r_1,q_2\}=0$,
\[
\{f,q_2\} = \{T,q_2\}(\phy_{q_1}^T)^* r_1.
\]
Thus, such an $f$ will give us a solution to both cohomological
equations if $T$ satisfies
\begin{equation}
  \label{eq-T}
  \begin{cases}
    \{T,q_1\} = 1 \\
    \{T,q_2\} = 0
  \end{cases}
\end{equation}
Notice that this can be easily understood in a geometrical manner near
a \emph{regular} point of the map $(q_1,q_2)$, in the following
way. We fix a hypersurface $\mathcal{T}_0$ transversal to the flow of
$q_1$ and invariant under the flow of $q_2$; then $T(z)$ is the
(locally unique) time such that $\phy_{q_1}^{-T(z)}(z)\in
\mathcal{T}_0$.

Here, we may take $T = \frac{1}{4} \left[ \ln|z_2|^2 - \ln|z_1|^2
\right]$, which corresponds to the hypersurface $\abs{z_1}=\abs{z_2}$,
and gives a smooth solution on each connected component of $\C^2
\setminus \{q =0 \}$ (singularities of $T$ are exactly the zero locus
of $q$).

The last thing one has to check is that this solution can be extended
to the whole space as a smooth solution; actually, $f$ will be flat on
the two complex axis $z_1 =0$ and $z_2 =0$. Let's denote a derivation
of arbitrary degree in the four variables $z_1,
\bar{z_1},z_2,\bar{z_2}$ by
\[
L_{\alpha \beta} :=
\frac{\partial^{\abs{\alpha}+\abs{\beta}}}{\partial
  z_1^{\alpha_1} \partial \bar{z_1}^{\beta_1} \partial
  z_2^{\alpha_2} \partial \bar{z_2}^{\beta_2}}.
\]
We can then write
\begin{equation}
  \begin{aligned}
    \label{equ:Lab}
    L_{\alpha \beta} f = & \int_0^T (L_{\alpha \beta} r_1) (e^s z_1, e^{-s} z_2) e^{ks} ds \\
    + & \left\langle \left( L_{\gamma \delta}T \right)^\ell
      \left(L_{\gamma' \delta'} r_1\right) (e^T z_1,e^{-T} z_2) e^{mT}
    \right\rangle
  \end{aligned}
\end{equation}
with $k \in \ZM$ depending on ${(\alpha, \beta)}$. The term between
brackets designates a finite sum of terms of the generical form inside
the bracket, where the number of terms and the exact values of
$\gamma,\delta,\gamma',\delta',\ell$ and $m$ depend on ${(\alpha,
  \beta)}$.

We will make use of the following fact~: if a smooth function
$r(z_1,z_2)$ is flat on $z_1=0$, then for any $N_1\geq 0$, the
function $r(z_1,z_2)/\abs{z_1}^{N_1}$ is still smooth and flat on
$z_1=0$. Of course the corresponding statement holds if $r$ is flat on
the plane $z_2=0$. Thus, when $r$ is flat on both planes $z_1=0$,
$z_2=0$, the function $r(z_1,z_2)/\abs{z_1}^N\abs{z_2}^{N_2}$ is again
smooth and flat on both planes. Therefore, given any $N_1$, $N_2$ and
any bounded region for $(z_1,z_2)$, there exists a constant $C>0$ such
that
\begin{equation}
  \abs{r(z_1,z_2)}\leq C \abs{z_1}^{N_1}\abs{z_2}^{N_2}.\label{equ:flat-flat}
\end{equation}

We return now to the terms in~\eqref{equ:Lab}.  Notice that, by
construction,
\[
\max(\abs{e^sz_1},\abs{e^{-s}z_2})\leq \max(\abs{z_1}, \abs{z_2}),
\]
when $s$ varies between $0$ and $T$. This is clear also from the
geometric picture.  Thus, using that $L_{\alpha\beta}r_1$ is flat
along $z_1=0$ and $z_2=0$ we have from~\eqref{equ:flat-flat}, for
bounded $(z_1,z_2)$ and for any $N_1,N_2$, a constant $C>0$ such that
\begin{align}
  \label{equ:flat-flat-bis}
  \abs{(L_{\alpha\beta}r_1)(e^{s} z_1, e^{-s} z_2)} & \leq
  C(e^{s}\abs{z_1})^{N_1}
  (e^{-s}\abs{z_2})^{N_2} \nonumber \\
  ~ & = C\abs{q} (e^{s}\abs{z_1})^{N_1-1} (e^{-s}\abs{z_2})^{N_2-1}
\end{align}

Choosing $N_1$ and $N_2$ large enough, and taking
$(\alpha,\beta)=(\gamma',\delta')$, $s=T$, this proves that each of
the terms between brackets~\eqref{equ:Lab} in tend to $0$ as
$\abs{q}\to 0$.

We consider now the integral term.  If $ |z_2| \leq |z_1|$, then
$T\leq 0$; letting $N_2=1$ in~\eqref{equ:flat-flat-bis} we get
\[
\forall s \in [T,0], \qquad \abs{(L_{\alpha\beta}r_1)(e^{s} z_1,
  e^{-s} z_2)} \leq C \abs{q} (e^{-\abs{s}}\abs{z_1})^{N_1-1}.
\]
If we suppose instead $ |z_1| \leq |z_2|$, we get similarly
\[
\forall s \in [0,T], \qquad \abs{(L_{\alpha\beta}r_1)(e^{s} z_1,
  e^{-s} z_2)} \leq C \abs{q} (e^{-\abs{s}}\abs{z_2})^{N_2-1}.
\]

Therefore we can write, for any bounded $(z_1,z_2)$ and any $N$,
\begin{equation}
  \abs{(L_{\alpha\beta}r_1)(e^{s}
    z_1, e^{-s} z_2)} \leq C \abs{q} (e^{-\abs{s}}\abs{z})^{N}, \quad \text{with
  } \abs{z}:=\max(\abs{z_1},\abs{z_2}).
  \label{equ:expo-estimate}
\end{equation}
Thus, with $N\geq \abs{k}+1$,
\[
\abs{\int_0^T (L_{\alpha \beta} r_1) (e^s z_1, e^{-s} z_2) e^{ks} ds }
\leq C'\abs{q} \int_0^{\abs{T}} e^{-s}ds \leq C'\abs{q}.
\]
This shows that this term tends to $0$ as $\abs{q}\to 0$ as well.

These estimates conclude the proof : our solution on $\C^2 \setminus
\{{q=0}\}$ extends to a smooth function on $\CM^2$. By continuity of
Poisson brackets, this extension is a solution to our cohomological
equation~\eqref{equ:eq-cohom-bis} on the whole space.
\end{proof}

\subsection{A Darboux lemma for focus-focus foliations}

Here again $\RM^4$ is endowed with the canonical symplectic form
$\omega_0$. Recall that the regular level sets of the map
$q=(q_1,q_2):\RM^4\to\RM^2$ are lagrangian for $\omega_0$.

\begin{prop}
  \label{prop:darboux}
  Let $\omega$ be a symplectic form on $\R^4$ such that

  \begin{itemize}
  \item[(a)] $\omega = \omega_0 + \error (\infty)$;
  \item[(b)] the regular level sets of $q$ are lagrangian for $\omega$;
  \item[(c)] for all $z\in\RM^4$,
    \[
    \int_{D_z} \omega-\omega_0 = 0,
    \]
    where $D_z$ is the disk given by
    \[
    D_z:=\{(\zeta z_1,\zeta z_2)\in\CM^2; \quad \zeta\in\CM,
    \abs{\zeta}\leq 1\}
    \]
    (here we identify $\RM^4$ with $\CM^2$ and denote
    $z=(z_1,z_2)\in\CM^2$).
  \end{itemize}
  Then there exists a local diffeomorphism $\Phi$ of $(\R^4,0)$ and
  $U=(U_1,U_2)$ a local diffeomorphism of $(\R^2,0)$ such that

  \begin{enumerate}
  \item $\Phi^* \omega = \omega_0$
  \item $\Phi^* q_1 = U_1(q_1,q_2)$
  \item $\Phi^* q_2 = U_2(q_1,q_2)$
  \item Both $\Phi$ and $U$ are infinitely tangent to the identity.
  \end{enumerate}

\end{prop}
 Notice that conditions 2. and 3. together mean that $\Phi$ preserves
  the (singular) foliation defined by the level sets of $q$. Notice
  also that the hypothesis (a),(b),(c) are in fact necessary~: for (a)
  and (b) this is obvious; for (c), remark that $\gamma_z:=\partial D_z$ is an
  orbit of the $S^1$-action generated by $q_2$ for the canonical
  symplectic form $\omega_0$, and thus is a
  homology cycle on the Liouville torus $q=\textup{const}$. Since
  $\Phi$ is tangent to the identity, $\Phi_*\gamma_z$ is homologous to
  $\gamma_z$; thus, if $\alpha_0$ is the Liouville 1-form on $\RM^4$,
  which is closed on the Liouville tori, we have
 \[
\int_{\gamma_z}\alpha_0 = \int_{\Phi_*(\gamma_z)}\alpha_0 = \int_{\gamma_z}\Phi^*\alpha_0,
\]
which by Stokes gives (c).
\begin{proof}[Proof of the proposition.]

  We use again the standard deformation method by Moser. Let
  \[
  \omega_s = (1-s)\omega_0 + s \omega.
  \]
  We look for $Y_s$ a time-dependant vector field defined for
  $s\in[0,1]$ whose flow $s\mapsto\varphi_{Y_s}^{s}$ satisfies
  $(\varphi_{Y_s}^{s})^* \omega_s = \omega_0$. Taking the derivative
  with respect to $s$ gives
  \[
  (\varphi_{Y_s}^{s})^* \left[ \frac{\partial \omega_s}{\partial s} +
    \mathcal{L}_{Y_s} \omega_s \right] = (\varphi_{Y_s}^{s})^* \left[
    \omega - \omega_0 + d(\imath_{Y_s}\omega_s) \right] = 0.
  \]
  $\omega$ and $\omega_0$ being closed, we can find, in a
  neighbourhood of the origin, smooth 1-forms $\alpha$ and $\alpha_0$
  such that $\omega = d\alpha$ and $\omega_0 = d\alpha_0$. Using the
  standard constructive proof of the Poincaré lemma, we can choose
  $\alpha$ and $\alpha_0$ such that $\alpha = \alpha_0 + \error
  (\infty)$. Let $\phy_{q_2}^t$ be the hamiltonian flow of $X^0_2$
  on $(\RM^4,\omega_0)$.

  Since $\omega_s(0)=\omega_0(0)=\omega_0$, one can find a
  neighbourhood of the origin on which $\omega_s$ is non-degenerate
  for all $s \in [0,1]$. This enables us to find a suitable $Y_s$ by
  solving
  \begin{equation}
    \label{eq-Y_s} 
    \imath_{Y_s} \omega_s = -(\alpha - \alpha_0) + df,
  \end{equation}
  for a suitable function $f$.  Here, any function $f$ such that
  $df(0)=0$ will yield a vector field $Y_s$ whose time-1 flow $\Phi$
  solves the point 1. of the lemma. It turns out that properly
  choosing $f$ will be essential in ensuring that $\Phi$ preserves the
  foliation ( point 2. and 3.).

  Let $X^0_1,X^0_2$ the hamiltonian vector fields associated to
  $q_1,q_2$ respectively, for $\omega_0$. Since the level sets of $q$
  are lagrangian for $\omega_0$, $X^0_1,X^0_2$ are commuting vector
  fields spanning the tangent space to regular leaves. Thus $\omega_0
  (X^0_1,X^0_2)=0$.  But, by assumption, the level sets of $q$ are
  lagrangian for $\omega$ as well. This implies that
  $\omega(X^0_1,X^0_2)=0$ as well. Thus $\omega_s(X^0_1,X^0_2)=0$ for
  all $s$~: the level sets of $q$ are lagrangian for $\omega_s$. This
  entails that the condition that $Y_s$ be tangent to the leaves can
  be written
  \begin{equation}
    \begin{cases}
      \omega_s(Y_s,X^0_1) = 0 \\
      \omega_s(Y_s,X^0_2) = 0.
    \end{cases}
    \label{equ:tangent}  
  \end{equation}
  We can expand this~:
  \[
  \eqref{equ:tangent} \Longleftrightarrow_{\eqref{eq-Y_s}}
  \begin{cases}
    -(\alpha - \alpha_0) (X^0_1) + df(X^0_1) = 0 \\
    -(\alpha - \alpha_0) (X^0_2) + df(X^0_2) = 0.
  \end{cases}
  \]
  Now we may let
  \[
  \begin{cases}
    r_1 := (\alpha - \alpha_0) (X^0_1) \\
    r_2 := (\alpha - \alpha_0) (X^0_2)
  \end{cases}
  \]
  and the condition becomes
  \[
  \eqref{equ:tangent} \Longleftrightarrow
  \begin{cases}
    \{f,q_1\} = r_1 \\
    \{f,q_2\} = r_2.
  \end{cases}
  \]
  (Here the Poisson brackets come from the canonical symplectic form
  $\omega_0$).  Notice that $r_1$ and $r_2$ are flat at the origin.

Next, recall the following formula for 1-forms : $d\alpha(X,Y) =
X\alpha(Y) - Y\alpha(X) - \alpha([X,Y])$. Thus
\[
0 = \omega_0 (X^0_1,X^0_2) = d\alpha_0(X^0_1,X^0_2) = X^0_1
\alpha_0(X^0_2) - X^0_2 \alpha_0(X^0_1) -
\alpha_0([X_1^0,X_2^0]),
\]
which implies
\[
\imath_{X^0_1} d(\alpha_0(X^0_2)) = \imath_{X^0_2} d(\alpha_0(X^0_1)).
\]
and similarly
\[
\imath_{X^0_1} d(\alpha(X^0_2)) = \imath_{X^0_2} d(\alpha(X^0_1)).
\]
Hence we may write the same equation again for $\alpha-\alpha_0$
which, in terms of $\omega_0$-Poisson brackets, becomes
\[
\{r_1,q_2\} = \{r_2,q_1\}.
\]
Therefore, a solution $f$ to this system~\eqref{equ:tangent} is
precisely given by the division lemma (Theorem~\ref{theo:division}),
provided we show that $r_2$ has vanishing $\phy_{q_2}^t$-average.
But, since $\frac{d}{dt}\phy_{q_2}^t = X_2^0(\phy_{q_2}^t)$, we have, 
\[
\forall z\in\RM^4, \quad \frac{1}{2\pi}\int_0^{2\pi}
r_2(\phy_{q_2}^t(z))dt = \int_{\gamma_z} \alpha-\alpha_0 =
\int_{D_z}\omega-\omega_0 = 0.
\]
Finally, we check that $Y_s$ as defined with \ref{eq-Y_s} vanishes at
the origin and hence yields a flow up to time 1 on a open
neighbourhood of the origin.

To conclude, the time-1 flow of $Y_s$ is a local diffeomorphism $\Phi$
that preserves the $q$-foliation and such that $\Phi^* \omega =
\omega_0$, which finishes the proof.

Notice that $Y_s$ is uniformly flat, whence $\Phi$ is a flat
perturbation of the identity.
\end{proof}

\section{ Proof of the main theorem }

We summarize here all the steps that bring us to prove
Theorem~\ref{theo-principal}.
\begin{align*}
  \begin{cases}
    F = \begin{pmatrix} f_1 \\ f_2 \end{pmatrix} \\
    \omega_0
  \end{cases} &
  \xrightarrow{\text{Lemma \ref{lemma-BirkP}}}
  \begin{cases}
    \chi^* F = G(q_1,q_2) + \error (\infty) \\
    \chi^* \omega_0 = \omega_0
  \end{cases}\\
&  \xrightarrow{\text{Theorem \ref{theo:morse}}}
  \begin{cases}
    \Upsilon^* G^{-1} (\chi^* F) = \begin{pmatrix} q_1 \\ q_2 \end{pmatrix} \\
    \Upsilon^* (\chi^* \omega_0) = \omega = \omega_0 + \error (\infty)
  \end{cases}\\
&  \xrightarrow{\text{Proposition \ref{prop:darboux}}}
  \begin{cases}
    \Phi^* \Upsilon^* G^{-1} (\chi^* F) = \begin{pmatrix} U(q_1,q_2) \\ U(q_1,q_2) \end{pmatrix} \\
    \Phi^* \Upsilon^* \chi^* \omega_0 = \omega_0
  \end{cases}
\end{align*}

Only the last implication needs an explanation~: indeed, even if the
Morse lemma is not symplectic, the initial foliation by $F$ is
lagrangian for $\omega_0$, and this implies that, under $\Upsilon$,
the target foliation by $q$ becomes lagrangian for the target
symplectic form $\omega$. Thus the hypothesis (a) and (b) of the
Darboux lemma~(Proposition \ref{prop:darboux}) are satisfied. That (c)
is also satisfied follows from the equivariance property
$(\mathcal{P})$ of Theorem~\ref{theo:morse}. Indeed, let $\alpha_0$ be
the Liouville 1-form in $\RM^4$, and
$\alpha:=\Upsilon^*\alpha_0$. Since $\Upsilon$ commutes with
$\phy_{q_2}^t$, we have
\[
\lie_{X_2^0}\alpha = \lie_{X_2^0}\Upsilon^*\alpha_0 = \Upsilon^*
 \lie_{X_2^0}\alpha_0.
\]
On the other hand, since $\imath_{X^0_2} d\alpha_0 = -dq_2$ and
(Lemma~\ref{lemm:homogeneous}) $d\imath_{X^0_2} \alpha_0 = dq_2$, we
get $\mathcal{L}_{X^0_2} \alpha_0 = 0$. Thus $\lie_{X_2^0}\alpha=0$
which, in turn, says that $d\imath_{\ham{q_2}^\omega} \alpha =
-\imath_{\ham{q_2}^\omega} d\alpha = dq_2$, where we denote by
$\ham{q_2}^\omega$ the $\omega$-gradient of $q_2$.  By property
$(\mathcal{P})$, $\ham{q_2}^\omega= X_2^0$, so $d\imath_{X_2^0} \alpha
= dq_2$.  Hence $\imath_{X^0_2} \alpha = q_2 + \beta$, where $\beta$ is
a constant, which is actually equal to 0 since $\imath_{X^0_2} \alpha
= q_2 + \error(\infty)$. We thus get
$\imath_{X_2^0}(\alpha-\alpha_0)=0$, which of course implies
\[
\int_{\gamma_z} \alpha-\alpha_0 = 0.
\]

Thus one may apply Proposition \ref{prop:darboux}, and the main
theorem~\ref{theo-principal} is shown for $\Psi := \Phi \circ \Upsilon
\circ \chi$ and $\tilde{G} := G \circ U$.

\bibliographystyle{plain}%
\bibliography{bibli-utf8}
\end{document}